\documentclass[a4paper,10pt]{amsart}
\usepackage{longtable}
\usepackage{amsmath}
\usepackage{amssymb}
\usepackage{amsthm}
\usepackage{hyperref}
\usepackage[retainorgcmds]{IEEEtrantools}
\newcommand{\bx}{\textnormal{Box}}
\newcommand{\R}{\ensuremath{\mathbb{R}}}
\newcommand{\Z}{\ensuremath{\mathbb{Z}}}

\DeclareMathOperator{\Ehr}{Ehr}
\DeclareMathOperator{\vol}{vol}
\DeclareMathOperator{\Eul}{Eul}

\newcommand{\codim}{\ensuremath{\textnormal{codim}}}

\newtheorem{Thm}{Theorem}[section]
\newtheorem*{Thm*}{Theorem}
\newtheorem{Def}[Thm]{Definition}
\newtheorem{Prop}[Thm]{Proposition}
\newtheorem*{Prop*}{Proposition}
\newtheorem{Lem}[Thm]{Lemma}
\newtheorem*{Lem*}{Lemma}
\newtheorem{Cor}[Thm]{Corollary}
\newtheorem*{Cor*}{Corollary}

\newtheorem{Conj*}{Conjecture}
\newtheorem{Quest}[Thm]{Question}

\theoremstyle{remark}
\newtheorem{Rem}[Thm]{Remark}

\newtheorem{Ex}[Thm]{Example}
\newtheorem*{Ex*}{Example}

\title{Unimodality questions for integrally closed lattice polytopes}

\author[Jan Schepers]{Jan Schepers}
\address{K.U.Leuven, Celestijnenlaan 200B, 3001 Leuven, Belgium \newline \indent \textit{E-mail addresses:} \textnormal{ \texttt{janschepers1@gmail.com}, \texttt{leen.vanlangenhoven@wis.kuleuven.be}}}

\author[Leen Van Langenhoven]{Leen Van Langenhoven}
\thanks{The first named author is a Postdoctoral Fellow of the Research Foundation - Flanders (FWO). \\ \indent
\textit{MSC2010 Subject Classification:} 52B20, 05A20, 11B68.}

\begin{document}

\begin{abstract}
It is a famous open question whether every integrally closed reflexive polytope has a unimodal Ehrhart $\delta$-vector. We generalize this question to arbitrary integrally closed lattice polytopes and we prove unimodality for the $\delta$-vector of lattice parallelepipeds. This is the first nontrivial class of integrally closed polytopes. Moreover, we suggest a new approach to the problem for reflexive polytopes via triangulations. 
\end{abstract}

\maketitle

\normalsize

\section{Introduction}\label{introduction}

By a \textit{lattice polytope} we mean the convex hull in $\R^m$ of finitely many points of the lattice $\Z^m$. We denote the number of lattice points in a lattice polytope $P$ by $\sharp P$. Assume that $P$ has dimension $d$. It is well known that the \textit{Ehrhart generating series} 
\[ \Ehr(P,t) =  1 + \sum_{n = 1}^{\infty} \sharp(nP)\, t^n \]
can be written as
\[   \frac{\delta_0 + \delta_1 t + \cdots + \delta_d t^d}{(1-t)^{d+1}},\]
for some nonnegative integers $\delta_i$, with $\delta_0=1$ and $\delta_d$ equal to the number of lattice points in the interior of $P$ (see for instance \cite{BeckRobins}). We denote the denominator by $\delta(P,t)$, although in the mirror symmetry literature one usually uses the notation $h_P^*(t)$ (see \cite{BatyrevNill}). The degree $s$ of $\delta(P,t)$ is called the \textit{degree} of $P$, and $l=d+1-s$ is called the \textit{codegree} of $P$. It is the smallest positive integer such that $lP$ has an interior lattice point. The vector 
\[ \delta(P) = (\delta_0, \delta_1 , \ldots , \delta_d)\] is called the \textit{$\delta$-vector} of $P$. By the \textit{volume} of $P$ we always mean the normalized volume, i.e., $\vol(P)$ is $d !$ times the Euclidean volume. We note that $\vol(P) = \delta_0+ \cdots + \delta_d$. A \textit{unimodular simplex} is a simplex of volume 1.\\

A very general question is to classify all possible $\delta$-vectors of lattice polytopes. Many restrictions in the form of inequalities are known (see \cite{Stapledon}). Our paper is motivated by the question whether every \textit{integrally closed reflexive} polytope has a \textit{unimodal} $\delta$-vector \cite{HibiOhsugiGorenstein}. Let us explain these notions. A polytope $P$ of dimension $d$ in $\R^d$ is called \textit{reflexive} if it contains the origin in its interior and if the dual set
\[  P^{\times} = \{ y \in \R^d\,|\, \langle x, y \rangle \geqslant -1 \text{ for all } x\in P  \}    \]
is again a lattice polytope, where $\langle \cdot , \cdot \rangle$ denotes the standard inner product. This means that all the supporting hyperplanes of facets of $P$ can be given by a linear equation with constant coefficient 1. Equivalently, $P$ is a translate of a reflexive polytope if and only if the $\delta$-vector of $P$ is symmetric, i.e., $\delta_i = \delta_{d-i}$ for all $i$ \cite{HibiCombinatorica}. 

A lattice polytope is called \textit{integrally closed} if every lattice point in a multiple $nP$ can be written as the sum of $n$ lattice points in $P$. We note that a face of an integrally closed polytope is integrally closed as well. Finally, a sequence $(\delta_0,\ldots , \delta_d)$ of real numbers is called \textit{unimodal} if 
\[ \delta_0 \leqslant \cdots \leqslant \delta_{c-1} \leqslant \delta_c \geqslant \delta_{c+1} \geqslant \cdots \geqslant \delta_d \]
for some $c$. Hibi showed that the $\delta$-vector of a reflexive polytope of dimension $d$ is unimodal if $d\leqslant 5$ and conjectured that this would hold in general \cite{Hibi}, but Musta\c{t}\u{a} and Payne gave counterexamples for dimension 6 and higher \cite{MustataPayne,Payne}. However, no integrally closed counterexample is known. More generally, one may ask the following question.

\begin{Quest}\label{Hoofdvraag}
Let $P$ be an integrally closed lattice polytope. Is it true that $\delta(P)$ is unimodal\,?
\end{Quest}

We answer this question affirmatively for \textit{lattice parallelepipeds} in Section~\ref{section2}. After unimodular simplices, these are the easiest examples of integrally closed polytopes. Moreover, we prove that lattice parallelepipeds with an interior lattice point satisfy a stronger unimodality property: they have an \textit{alternatingly increasing} $\delta$-vector (see Definition~\ref{AlternatinglyIncreasing}). 
In Section~\ref{section3} we prove that integrally closed polytopes up to dimension 4 have a unimodal $\delta$-vector. Finally, in Section~\ref{ReflexivePolytopes} we turn our attention to reflexive polytopes again. We relate the unimodality question to the existence of certain triangulations, which we call \textit{box unimodal triangulations}. Integrally closed reflexive polytopes whose boundary admits such a triangulation have a unimodal $\delta$-vector. We remark that these triangulations exist for all polytopes up to dimension 4; an immediate corollary of this fact is Hibi's result that reflexive polytopes up to dimension 5 have a unimodal $\delta$-vector.

For other recent work on integrally closed polytopes (e.g.\ a bound on the volume) we refer to \cite{Hegedus}.

\subsection*{Acknowledgements} 
We would like to thank Sam Payne for a helpful discussion concerning Section~\ref{ReflexivePolytopes}. The computer programs Normaliz \cite{Normaliz} and TOPCOM \cite{TOPCOM} were very useful for experimenting with examples.

\section{The \texorpdfstring{$\delta$}{delta}-vector of a lattice parallelepiped}\label{section2}

In this section we illustrate Question~\ref{Hoofdvraag} by proving that the $\delta$-vector of a lattice parallelepiped is unimodal. Lattice parallelepipeds are the easiest examples of integrally closed polytopes after unimodular simplices. Moreover, we will show that a strong unimodality property holds if the parallelepiped has at least one interior lattice point.

Throughout this section we assume that $\lbrace v_0, v_1, \ldots , v_r \rbrace$ is a set of $r+1$ linearly independent vectors with integer coordinates in $\R^m$. We write $P = \langle v_0, v_1, \ldots , v_r\rangle$ for the $r$-dimensional simplex with the $v_i$ as vertices. Let $\lozenge P$, $\Pi P$ and $\bx(P)$ denote respectively the closed, half-open and open parallelepiped spanned by the vertices of $P$:
\[ \lozenge P  = \Big\{\sum_{i=0}^r \lambda_i v_i \mid 0\leqslant\lambda_i\leqslant 1 \Big\}, \]
\[ \Pi(P) = \Big\{\sum_{i=0}^r \lambda_i v_i \mid 0\leqslant\lambda_i < 1 \Big\}, \]
\[ \bx(P) =\Big\{\sum_{i=0}^r \lambda_i v_i \mid 0 <\lambda_i< 1 \Big\}. \]
It is well known that lattice parallelepipeds are integrally closed. This can be proven in an elementary way. For a more general statement, see \cite[Lemma 2.4]{GubeladzeLongEdges}. 

We will use the notation $b(P)$ for $\sharp \bx(P)$. The empty set is considered to be a simplex of dimension $-1$. We put $\lozenge \, \emptyset = \Pi (\emptyset) =  \bx(\emptyset) = \{0\}$ and hence $b(\emptyset) = 1$. 

\begin{Lem}\label{number of elements} 
$$\sharp (n \lozenge P) = \sum_{\emptyset\subseteq F \subseteq P} n^{\dim F+1} \sharp \Pi(F),$$
where the sum runs over all faces $F$ of $P$.
\end{Lem}

\begin{proof}
We have 
$$\sharp (n\lozenge P ) = \sum_{\substack{F\subseteq n\lozenge P\\ 0 \leqslant \dim F } } \sharp (F^\circ), $$ where the sum runs over all nonempty faces of $n\lozenge P$ and where $F^\circ$ denotes the relative interior of a face.

Let $F$ be a face of $n\lozenge P$ of codimension $k$ with $0\leqslant k\leqslant \dim( n\lozenge P) = r+1$. Then there exist fixed $\lambda_{i_1},\ldots \lambda_{i_k} \in \{0,1\}$ such that
\[  F  = \Big\{\sum_{i=0}^r \lambda_i n v_i \mid 0\leqslant\lambda_i\leqslant 1  \textnormal{ if }\lambda_i \neq \lambda_{i_1},\ldots, \lambda_{i_k} \Big\}. \] 
We find that $F$ is a translate of $\lozenge n G$, where $G$ is a face of $P$ of codimension $k$. There are $2^k$ translates of $\lozenge n G$ occurring as faces of $n\lozenge P$, one for each choice of the $\lambda_{i_j}$. Hence 
\begin{eqnarray*}
\sharp (n\lozenge P)& = & \sum_{\emptyset\subseteq G\subseteq P} 2^{\codim(G,P)} b(nG)\\
&=& \sum_{\emptyset\subseteq G\subseteq P} \sum_{G\subseteq G'\subseteq P}b(nG)\\
&=& \sum_{\emptyset\subseteq G'\subseteq P} \sum_{\emptyset\subseteq G\subseteq G'} b(nG)\\
&=& \sum_{\emptyset\subseteq G'\subseteq P} \sharp \Pi(nG')\\
&=& \sum_{\emptyset\subseteq G'\subseteq P} n^{\dim G'+1}\,\sharp\Pi (G'). 
\end{eqnarray*}
In the last step we used that $\Pi(nG')$ is covered by $n^{\dim G' +1}$ translates of $\Pi(G')$.
\end{proof}

Before we continue, we recall the definition of the \textit{Eulerian polynomials}. Let $n\geq 0$ be an integer. The Eulerian polynomial $\Eul(n,t)$ of degree $n$ is given by
\[  \Eul(n,t) = \sum_{i=0}^n t^i \sum_{j = 0}^i (-1)^j \binom{n+2}{j} (i+1-j)^{n+1}. \]
The coefficients of these polynomials are called the \textit{Eulerian numbers}. These numbers also occur in the $\delta$-vectors of hypercubes: $\Eul(n,t) = \delta(\ell^{n +1},t)$, where $\ell$ denotes the standard 1-dimensional simplex $\langle 0,1\rangle$. See \cite[Thm.\ 2.1]{BeckRobins} for a proof. We define $\Eul(-1,t) := 1/t$.

\begin{Prop}\label{delta par P}
We have the following equalities for $\delta(\lozenge P, t)$:
\begin{eqnarray*}
\lefteqn{\delta(\lozenge P, t)}\\
& = &\sum_{\emptyset\subseteq G \subseteq P} b(G) \sum_{n=0}^\infty t^n\,n^{\dim G+1}(n+1)^{\dim P-\dim G}(1-t)^{\dim P +2}\\
&=&\sum_{\emptyset\subseteq G \subseteq P} b(G) \sum_{G\subseteq F \subseteq P} t\,(1-t)^{\dim P -\dim F }\Eul(\dim F,t).
\end{eqnarray*}
\end{Prop}

\begin{proof}
By Lemma~\ref{number of elements} we have
\begin{eqnarray*}
\delta(\lozenge P, t)&=& (1-t)^{\dim \lozenge P+1} \Ehr(\lozenge P, t)\\
&=& (1-t)^{\dim P+2}\sum_{n= 0}^\infty t^n \sum_{\emptyset\subseteq F\subseteq P} n^{\dim F+1}\sharp(\Pi(F))\\
&=&(1-t)^{\dim P+2}\sum_{n= 0}^\infty t^n \sum_{\emptyset\subseteq F\subseteq P} n^{\dim F+1}\sum_{\emptyset\subseteq G\subseteq F} b(G)\\
&=&(1-t)^{\dim P+2}\sum_{\emptyset\subseteq G\subseteq P}  b(G) \sum_{n= 0}^\infty     t^n\sum_{G\subseteq F\subseteq P}n^{\dim F+1}. 
\end{eqnarray*}
We compute $\delta(\lozenge P,t)$ in two ways. The first way is:
\begin{eqnarray*}
\lefteqn{\delta(\lozenge P, t)}\\
&=&(1-t)^{\dim P+2}\sum_{\emptyset\subseteq G\subseteq P}  b(G) \sum_{n= 0}^\infty     t^n \sum_{k=0}^{\dim P - \dim G}\binom{\scriptstyle{\dim P-\dim G}}{k}n^{\dim G+k+1}\\
&=& (1-t)^{\dim P+2}\sum_{\emptyset\subseteq G\subseteq P} b(G) \sum_{n= 0}^\infty     t^n n^{\dim G+1}(n+1)^{\dim P-\dim G}.
\end{eqnarray*}

The second way is:
\begin{eqnarray*}
\lefteqn{\delta(\lozenge P, t)}\\
&=& (1-t)^{\dim P+2}\Big(\hspace{-1mm} \sum_{\emptyset\subseteq G\subseteq P}  b(G)\sum_{\substack{G\subseteq F\subseteq P \\ 0\leqslant \dim F}} \sum_{n= 0}^\infty t^n n^{\dim F +1} + b(\emptyset)\sum_{n= 0}^\infty t^n \Big)\\
&=& (1-t)^{\dim P+2}\Big(\hspace{-1mm} \sum_{\emptyset\subseteq G\subseteq P}  b(G)\sum_{\substack{G\subseteq F\subseteq P\\0\leqslant \dim F}} t \sum_{m= 0}^\infty t^m (m+1)^{\dim F +1} + \frac{b(\emptyset)}{1-t} \Big) \\
&=&\sum_{\emptyset\subseteq G\subseteq P}  b(G)\sum_{\substack{G\subseteq F\subseteq P\\0\leqslant \dim F}} t\, (1-t)^{\dim P-\dim F} \delta(\ell^{\dim F +1},t) + b(\emptyset)(1-t)^{\dim P +1}\\
&=&\sum_{\emptyset\subseteq G\subseteq P} b(G) \sum_{G\subseteq F\subseteq P} t\,(1-t)^{\dim P-\dim F} \Eul(\dim F, t).
\end{eqnarray*}
\end{proof}

\noindent Let us write $\delta(\lozenge P, t) = \sum_{\emptyset\subseteq G \subseteq P} b(G)\,A(\dim P, \dim G, t)$
where
 \begin{align*}
A(i, j, t)&= (1-t)^{i+2} \sum_{n=0}^\infty       t^n n^{j+1}(n+1)^{i-j} \tag{$\clubsuit$}\label{A1}\\ 
&=\sum_{k=0}^{i-j} t\,(1-t)^{i - j-k } \binom{i-j}{k} \Eul(j+k,t) .\tag{$\heartsuit$} \label{A2}
\end{align*}
for integers $i\geqslant j\geqslant -1$.
 
To get an idea of how these polynomials look like, let us write them down for $i=4$:
\setlongtables
\begin{longtable}{c|c}
&  
\\ [-0.3cm] $j$ & $A(4,j,t)$   
\\ [0.2cm] \hline &  
\\ [-0.3cm] $-1$ & $1 + 26t + 66t^2 + 26t^3 + t^4$
 \\ [0.3cm] \hline  &  
 \\ [-0.3cm] 0 & $16t + 66t^2 + 36t^3 + 2t^4$
 \\ [0.3cm] \hline &  
 \\ [-0.3cm] 1 & $8t + 60t^2 + 48t^3 + 4t^4$
 \\ [0.3cm]  \hline  &    
 \\ [-0.3cm] 2 & $ 4t + 48t^2 + 60t^3  + 8t^4 $
 \\ [0.3cm] \hline &  
 \\ [-0.3cm] 3 & $2t + 36t^2 + 66t^3 + 16t^4$
 \\ [0.3cm] \hline &  
\\ [-0.3cm] $4$ & $\quad \ t + 26t^2 + 66t^3 + 26t^4 + t^5$
 \\ [0.1cm]
\end{longtable}

To prove that the $\delta$-vector of a parallelepiped is unimodal we will need that the coefficients of these $A$-polynomials are unimodal. This will be done in Proposition~\ref{h*unimodal}. First we derive some basic properties of the $A$-polynomials.

\begin{Lem}\label{mirror with} 
For all $j$ where $0\leqslant j \leqslant i-1$:
$$ t^{i+1} A(i,j,t^{-1})= A(i, i - 1  -j  ,t).$$
\end{Lem}

\begin{proof}
By the properties of the Eulerian numbers we have for every $k\geqslant 0$ that $t^k\Eul(k,t^{-1})=\Eul(k,t)$. We use this, equations (\ref{A1}) and (\ref{A2}) for $A(i,j,t)$, and the equality $\Eul(k,t) = \delta(\ell^{k+1},t)$ to prove the lemma:
\begin{IEEEeqnarray*}{lcl}
t^{i+1}A(i, j, t^{-1}) &\ =\ &t^{i+1}\sum_{k=0}^{i-j} t^{-1}  (1-t^{-1})^{i-j-k} \binom{i-j}{k} \Eul(j+k,t^{-1})\\
&=\ &\sum_{k=0}^{i-j} (t-1)^{i-j-k} \binom{i-j}{k} \Eul(j+k,t) \\
&=\ &\sum_{k=0}^{i-j} (t-1)^{i-j-k} \binom{i-j}{k} (1-t)^{j+k + 2}  \sum_{n= 0}^\infty  (n+1)^{j+k+1}t^n\\
&=\ &(1-t)^{i+2}\sum_{k=0}^{i-j} (-1)^{i-j-k} \binom{i-j}{k} \sum_{n= 0}^\infty     (n+1)^{j+k+1}t^n\\
&=\ &(1-t)^{i+2}\sum_{n= 0}^\infty     (-1)^{i-j} t^n (n+1)^{j+1} \sum_{k=0}^{i - j} \binom{i-j}{k} (-1)^{k} (n+1)^k\\
&=\ &(1-t)^{i+2}\sum_{n= 0}^\infty     (-1)^{i-j}t^n (n+1)^{j+1} (-(n+1)  + 1)^{i -j}\\
&=\ &(1-t)^{i+2}\sum_{n= 0}^\infty     t^n  n^{i -j}(n+1)^{j+1}\\
&=\ & A(i, i -1-j,t).
\end{IEEEeqnarray*}
\end{proof}

\begin{Rem}\label{graad}
For $0\leqslant j \leqslant i-1$ it follows from equation (\ref{A2}) that the degree of $A(i, j, t)$ is at most $i+1$. We also see that $t$ divides $A(i, j, t)$. Hence, from the previous lemma, we find that $\deg A(i, j, t)\leqslant i$. We will use this in Theorem~\ref{LatticeParUnimodal}.
\end{Rem}

\begin{Lem}\label{A(k,t)=delta(triangle line)}
The $A$-polynomials have the following properties:
\begin{enumerate}
\item $A(i,-1, t)= \delta(\ell^{i + 1},t),$
\item $A(i,i,t)=t\,\delta(\ell^{i+1},t)$ for $i\geqslant 0$ and
\item for all $0\leqslant j < i$ such that $2j+1 \geqslant i$ we have:
\[A(i,j,t)= 2^{i-j}\,t\,\delta(\triangle^{i-j}\times \ell^{2j+1-i},t),\]
where $\triangle$ denotes the standard 2-simplex $\langle (0,0),(1,0),(0,1) \rangle$.
\end{enumerate}
\end{Lem}

\begin{proof}
\begin{enumerate}
\item Using equation (\ref{A1}) for the $A$-polynomial, we see that 
$$A(i, -1,t)=\sum_{n=0}^\infty t^n\,(n+1)^{i+1}(1-t)^{i+2},$$
which is equal to $\delta(\ell^{i +1},t)$.
\item This is clear from equation (\ref{A2}).
\item Note that $\sharp (m\triangle) =(m+1)(m+2)/2$. When $2j+1\geqslant i > j\geqslant 0$, we have
\begin{IEEEeqnarray*}{rcl} 
A(i,j,t) &\ =\ & (1-t)^{i+2}\sum_{n= 0}^\infty t^n  n^{j+1} (n+1)^{i -j}\\
&=\ & (1-t)^{i+2}\sum_{n= 0}^\infty t^n n^{i -j}(n+1)^{i -j} n^{2j+1-i}\\
&=\ & (1-t)^{i+2}\sum_{m= 0}^\infty t^{m+1}(m+1)^{i -j}(m+2)^{i -j}(m+1)^{2j+1-i}\\
&=\ & 2^{i -j} t (1-t)^{i+2}\sum_{m= 0}^\infty t^m \sharp \big(m (\triangle^{i -j}\times \ell^{2j+1-i})\big)\\
&=\ & 2^{i -j} t\,\delta(\triangle^{i -j}\times \ell^{2j+1-i},t).
\end{IEEEeqnarray*}
\end{enumerate}
\end{proof}

From the previous lemma we see that the $A$-polynomials are related to $\delta$-vectors of Cartesian products of $\ell$ and $\triangle$. To prove unimodality of the coefficients of the $A$-polynomials, we will study in general what happens to a $\delta$-vector after taking the Cartesian product with $\ell$.

\subsection{Cartesian product with \texorpdfstring{$\ell$}{ell}}

\begin{Lem}\label{delta(Pxl)}
Let $P$ be a lattice polytope of dimension $d\geqslant 0$ with $\delta$-vector $\delta(P)=(h_0,h_1,\ldots,h_d)$.
Then the $\delta$-vector of the Cartesian product of $P$ with $\ell$ equals $\delta(P\times \ell)=(\delta_0,\delta_1,\ldots,\delta_d,0)$, where 
\[ \delta_i=(i+1)h_i+(d+1-i)h_{i-1}. \] 
Here we take $h_{d+1}=h_{-1}=0$. In particular, $\deg (P \times \ell) = \deg P$ if $\deg P = d$ and $\deg (P \times \ell) = \deg P + 1$ if $\deg P < d$.
\end{Lem}

\begin{proof}
Since $(h_0+h_1t+\cdots+h_dt^d)/(1-t)^{d+1}=\Ehr(P,t)=\sum_{n= 0}^\infty \sharp (nP) t^n$, one computes that $\sharp (nP) =\sum_{k=0}^n h_{n-k}\binom{d+k}{k}.$
And since $\sharp (n\ell)=n+1$, we find that $\sharp (n (P\times \ell))=\sharp (nP)\times \sharp(n\ell)= \sum_{k=0}^n h_{n-k}\binom{d+k}{k} (n+1)$. This way we can find $\delta(P\times \ell)$ in terms of the coefficients of $\delta(P)$:
\begin{eqnarray*}
\delta(P\times \ell,t) & = & \Ehr(P\times \ell,t) (1-t)^{d+2}\\
&=& \sum_{n= 0}^\infty       t^n \sharp (n(P\times \ell))\sum_{j=0}^{d+2}(-1)^j\binom{d+2}{j}t^j\\
&=& \sum_{i= 0}^{\infty} t^i \sum_{m=0}^i\sharp \big((i-m)(P\times \ell)\big)(-1)^m\binom{d+2}{m}.
\end{eqnarray*}
Hence for $0\leqslant i \leqslant d+1$ we find
\begin{eqnarray*}
\delta_i &=&\sum_{m=0}^i\sharp\big((i-m)(P\times \ell)\big)(-1)^m\binom{d+2}{m}\\
&=&\sum_{m=0}^i (-1)^m\binom{d+2}{m} (i-m+1) \sum_{j=0}^{i-m} h_{i-m-j}\binom{d+j}{j} \\
&=&\sum_{m=0}^i  (-1)^m\binom{d+2}{m} (i-m+1)\sum_{k=m }^ih_{i-k}\binom{d+k-m}{k-m}\\
&=&\sum_{k=0}^i h_{i-k}\sum_{m=0}^k(-1)^m (i-m+1)\binom{d+2}{m}\binom{d+k-m}{k-m}.
\end{eqnarray*}

One then easily sees that in this last sum the coefficient of $h_i$ equals $i+1$ and the coefficient of $h_{i-1}$ equals $d+1-i$. To end this proof we need to verify that for $k\geqslant 2$:
$$\sum_{m=0}^k (-1)^m (i-m+1)\binom{d+2}{m}\binom{d+k-m}{k-m}=0.$$
This follows from Lemma \ref{sum0} below.
\end{proof}

\begin{Lem}\label{sum0}
Let $d\geqslant 0$ be an integer. For $k\geqslant 2$ we have:
\begin{enumerate}
\item $c_k= \sum_{m=0}^k(-1)^m \binom{d+2}{m}\binom{d+k-m}{k-m}=0,$
\item $b_k= \sum_{m=0}^k(-1)^m m\binom{d+2}{m}\binom{d+k-m}{k-m}=0.$
\end{enumerate}
Here we use the standard convention that $\binom{a}{b} = 0$ if $b>a$.
\end{Lem}

\begin{proof}
\begin{enumerate}
\item The first part of this lemma follows from the next equalities:
\begin{eqnarray*} 
1-t &=& \frac{(1-t)^{d+2}}{(1-t)^{d+1}}\\
 &=& \sum_{s=0}^{d+2}t^s(-1)^s\binom{d+2}{s}\sum_{n=0}^{\infty     }t^n\binom{d+n}{n}\\
&=&\sum_{k=0}^\infty      t^k\sum_{m=0}^k(-1)^m \binom{d+2}{m}\binom{d+k-m}{k-m}.
\end{eqnarray*}
So $c_k$ is the $k$-th coefficient in this expansion of $1-t$.
\item We prove the second part by using induction on $k$ and part (1) of this lemma.
The claim $b_k=0$ is clear for $k=2$. So now we assume $k\geqslant 3$ and $b_{k-1}=0$. Then we have:
\begin{eqnarray*}
b_k&=& \sum_{m=0}^k(-1)^m m\binom{d+2}{m}\binom{d+k-m}{k-m}\\
&=& \sum_{m=1}^k(-1)^m (d+3-m)\binom{d+2}{m-1}\binom{d+k-m}{k-m}\\
&=& (-1)\sum_{n=0}^{k-1}(-1)^n (d+2-n)\binom{d+2}{n}\binom{d+(k-1)-n}{(k-1)-n}\\
&=&-(d+2)c_{k-1}+b_{k-1}\\
&=&0.
\end{eqnarray*} 
\end{enumerate}
\end{proof}

\begin{Cor}\label{delta iplus1-delta i} With the same notation as above: \begin{enumerate}
\item $\delta_{i+1}-\delta_i=(i+2)(h_{i+1}-h_i)+(d+1-i)(h_i-h_{i-1})$ and 
\item in particular, 
\begin{enumerate}
\item whenever $h_0\leqslant h_1\leqslant \cdots\leqslant h_r$, then $\delta_0\leqslant \delta_1\leqslant \cdots\leqslant \delta_r$ and
\item whenever $h_r\geqslant h_{r+1}\geqslant\cdots\geqslant h_s$, where $s=\deg P$, then $\delta_{r+1}\geqslant \delta_{r+2}\geqslant\cdots\geqslant \delta_{s+1}.$
\end{enumerate}
\end{enumerate}
\end{Cor}

\begin{proof}
\begin{enumerate}
\item This follows immediately from Lemma \ref{delta(Pxl)}.
\item These statements follow immediately from part (1) of this corollary.
\end{enumerate}
\end{proof}

In Theorem~\ref{delta(triangle k line i)} we will prove that the following strong unimodality property is preserved for polytopes with codegree 2 or 3 when taking the Cartesian product with $\ell$. 

\begin{Def}\label{AlternatinglyIncreasing}
Let $h =(h_0,\ldots  ,h_s)$ be a sequence of real numbers. We say that $h$ is \textnormal{alternatingly increasing} if
\[ h_0 \leqslant h_s \leqslant h_1 \leqslant h_{s-1} \leqslant \cdots \leqslant h_{\lfloor(s-1)/2\rfloor}\leqslant h_{s-\lfloor(s-1)/2\rfloor}\leqslant h_{\lfloor (s+1)/2\rfloor}. \]
\end{Def}

We will slightly abuse this terminology in the following way: if $P$ is a polytope of dimension $d$ and degree $s<d$ with $\delta$-vector $(\delta_0,\ldots,\delta_s,0,\ldots,0)$, then we will say that $\delta(P)$ is alternatingly increasing if $(\delta_0,\ldots,\delta_s)$ is.

\begin{Cor}\label{alternatingly increasing}
Let $P$ be a lattice polytope of dimension $d\geqslant 1$ and degree $s < d$.  
Write $\delta(P)=(h_0,\ldots,h_s,0,\ldots,0)$.
\begin{enumerate}
\item If $h_0\leqslant h_1\leqslant\cdots\leqslant h_{\lfloor (s+1)/2\rfloor}$ and $h_i\leqslant h_{s-i}$ for all $0\leqslant i\leqslant \lfloor s/2\rfloor$, then 
the same is true for $\delta(P\times\ell)$, i.e., with $s'=s+1=\deg (P\times \ell)$: 
\begin{eqnarray*}
\delta_i\leqslant \delta_{s'-i}, \textit{ for } 0\leqslant i\leqslant \lfloor s'/2\rfloor
&\textit{ and}\\
\delta_0\leqslant \delta_1\leqslant \cdots\leqslant \delta_{\lfloor(s'+1)/2\rfloor}.&
\end{eqnarray*}
\item If $s \geqslant d-2 $ and $h_{s+1-i}\leqslant h_i$ for $1\leqslant i\leqslant\lfloor(s+1)/2\rfloor$ and $h_{\lfloor(s+1)/2\rfloor}\geqslant \cdots\geqslant h_s$, then the same is true for $\delta(P\times\ell)$, i.e., with $s'=s+1=\deg (P\times \ell)$:
\begin{eqnarray*}\delta_{s'+1-i} \leqslant \delta_i \textit{ for } 1\leqslant i \leqslant \lfloor (s'+1)/2\rfloor & \textit{ and}\\
\delta_{\lfloor(s'+1)/2\rfloor}\geqslant  \cdots\geqslant \delta_{s'}.& 
\end{eqnarray*}
\item \label{increase 3} These two results combined give the following: when $s \geqslant d-2$ and $\delta(P)$ is alternatingly increasing, then $\delta(P\times \ell)$ is also alternatingly increasing. 
\end{enumerate}
\end{Cor}

\begin{proof}
\begin{enumerate} 
\item This statement follows from Corollary~\ref{delta iplus1-delta i}(2)(a) and Lemma~\ref{delta(Pxl)}: 
\begin{eqnarray*} 
\delta_{s+1-i}-\delta_i & = & (s+2-i)h_{s+1-i}+(d+1-(s+1-i))h_{s-i}\\
& & -(i+1)h_i-(d+1-i)h_{i-1}\\
&=& (s+2-i)(h_{s+1-i}-h_{i-1})+(i+1)(h_{s-i}-h_i)\\
&&+(d-s-1)(h_{s-i}-h_i)+(d-s-1)(h_i-h_{i-1}).
\end{eqnarray*}
\item The statement follows from Corollary \ref{delta iplus1-delta i}(2)(b) and Lemma~\ref{delta(Pxl)}:
\begin{eqnarray*} 
\delta_i-\delta_{s+2-i} & = & (i+1)h_i+(d+1-i)h_{i-1}\\
& & -(s+3-i)h_{s+2-i}-(d+1-(s+2-i))h_{s+1-i}\\
&=& (i+1)(h_i-h_{s+1-i})+(d+1-i)(h_{i-1}-h_{s+2-i})\\
&&+(s+2-d)(h_{s+1-i}-h_{s+2-i}).
\end{eqnarray*}
\item This follows from $(1)$ and $(2)$ of this corollary.
\end{enumerate}
\end{proof}

For the proof of the following theorem we need the notions of \textit{special simplex} and \textit{compressed polytope}. We refer to Athanasiadis' paper \cite{Athanasiadis} for the definitions.

\begin{Thm}\label{delta(triangle k line i)} 
When $P$ is a polytope such that $\delta(P)$ is alternatingly increasing with codegree $l = 2$ or $3$, then for all $i\geqslant 0$ we have that $\delta(P\times \ell^i)$ is also alternatingly increasing of the same codegree. In particular, for every $k,i\geqslant 0$ we have that $\delta(\triangle^k\times \ell^i)$ is alternatingly increasing. Especially, $\delta(\triangle^k\times \ell^i)$ is unimodal, with a `top' in $\delta_{\lfloor (2k+i-1)/2 \rfloor}$ if $(k,i)\neq (0,0)$.
\end{Thm}
\begin{proof}
The first claim follows immediately from Corollary~\ref{alternatingly increasing}(\ref{increase 3}) by induction on $i$.

Next we look at the polytope $\triangle^k\times\ell^i$. First let $k$ be zero. Since $\delta(\ell)=(1,0)$ it follows from Corollary~\ref{alternatingly increasing}(\ref{increase 3}) that $\delta(\ell^i)$ is alternatingly increasing for all $i$. Moreover, $\ell^i$ has codegree 2 if $i>0$. From Definition~\ref{AlternatinglyIncreasing} we see that the top of $\delta(\ell^i)$ lies at $\delta_{\lfloor i/2 \rfloor}$. If $i$ is even, this equals $\delta_{\lfloor (i-1)/2 \rfloor}$ by the symmetry of the Eulerian numbers.

Now take $k\geqslant 1$. It suffices to prove that $\delta(\triangle^k)$ is alternatingly increasing of codegree $c = 3$. It is easy to see that $\triangle^k$ is a $(0,1)$-polytope that satisfies a system of inequalities as in the main theorem of \cite{HibiOhsugi}, so it follows that $\triangle^k$ is compressed. It also has the special simplex $\Sigma=\left\langle ( 0,0,\ldots,0,0 ),(1,0,\ldots,1,0),(0,1,\ldots,0,1) \right\rangle$. So from \cite[Thm.\ 3.5]{Athanasiadis} it follows that $\delta(\triangle^k)$ is symmetric and unimodal of codegree $l=3$, hence also alternatingly increasing. The fact that the top of $\delta(\triangle^k \times \ell^i)$ lies at $\delta_{\lfloor (2k+i-1)/2 \rfloor}$ if $k\neq 0$ simply follows from Definition~\ref{AlternatinglyIncreasing}. 
\end{proof}

\begin{Rem}
Note that $\triangle^k$ is a \textit{Gorenstein polytope} of index 3. This means by definition that $3\triangle^k$ is a translate of a reflexive polytope. So, instead of Athanasiadis' result, we also could have used \cite[Thm.\ 1]{BrunsRomer}.
\end{Rem}




Finally we return to the unimodality questions.

\begin{Prop}\label{h*unimodal} 
For every $-1 \leqslant j \leqslant i$ we have that the coefficients of $A(i,j,t)$ form a unimodal sequence.
\end{Prop}

\begin{proof}
This follows from Lemma~\ref{mirror with}, Lemma~\ref{A(k,t)=delta(triangle line)} and Theorem~\ref{delta(triangle k line i)}.
\end{proof}

\begin{Thm}\label{LatticeParUnimodal} 
The $\delta$-vector of a lattice parallelepiped is unimodal.
\end{Thm}

\begin{proof}
Let $\lozenge P$ be a lattice parallelepiped spanned by the vertices of a simplex $P$. We use Proposition~\ref{delta par P}, Lemma~\ref{mirror with}, Lemma~\ref{A(k,t)=delta(triangle line)} and Theorem~\ref{delta(triangle k line i)} in the following, as well as Remark~\ref{graad}.

When $\dim P$ is odd (and positive), then
\begin{enumerate}
\item $A(\dim P,-1,t)= a_0 + a_1t + \cdots + a_{\dim P}t^{\dim P}$ with\\
$1 = a_0 \leqslant a_1 \leqslant \cdots\leqslant a_{\lfloor \dim P/2\rfloor}=a_{\lfloor \dim P/2\rfloor+1}\geqslant\cdots\geqslant a_{\dim P-1}\geqslant a_{\dim P}$,
\item $A(\dim P,\dim P,t)=tA(\dim P,-1,t)=0+a_1t+\cdots+a_{\dim P+1}t^{\dim P+1}$. So\\
$a_0=0\leqslant a_1\leqslant\cdots\leqslant a_{\lfloor \dim P/2\rfloor+1}=a_{\lfloor \dim P/2 \rfloor +2}\geqslant \cdots\geqslant a_{\dim P +1}$ and
\item for all $k$ between $0$ and $\dim P-1$, we have $A(\dim P,k,t)=0+a_1t+\cdots+a_{\dim P}t^{\dim P}$ where\\
$a_0=0\leqslant a_1\leqslant\cdots\leqslant a_{\lfloor \dim P/2\rfloor+1} \geqslant a_{\lfloor \dim P/2\rfloor +2}\geqslant\cdots\geqslant a_{\dim P}$.
\end{enumerate}
Hence $\delta(\lozenge P,t)$ is also unimodal.

When $\dim P$ is even, then
\begin{enumerate}
\item $A(\dim P,-1,t)= a_0 + a_1t + \cdots + a_{\dim P}t^{\dim P}$ with\\
$1 = a_0 \leqslant a_1 \leqslant \cdots\leqslant a_{ \dim P/2} \geqslant a_{ \dim P/2 +1}\geqslant\cdots\geqslant a_{\dim P-1}\geqslant a_{\dim P}$,
\item $A(\dim P,\dim P,t)=tA(\dim P,-1,t)=0+a_1t+\cdots+a_{\dim P+1}t^{\dim P+1}$. So \\
$a_0=0\leqslant a_1\leqslant\cdots\leqslant a_{\dim P/2+1}\geqslant a_{ \dim P/2  +2}\geqslant \cdots\geqslant a_{\dim P +1}$,
\item for all $k$ between $0$ and $\dim P/2-1$, we have $A(\dim P, k ,t)=0+a_1t+\cdots+a_{\dim P}t^{\dim P}$ where\\
$a_0=0\leqslant a_1\leqslant\cdots\leqslant a_{ \dim P/2} \geqslant a_{ \dim P/2 +1}\geqslant\cdots\geqslant a_{\dim P}$ and
\item for all $k$ between $\dim P/2$ and $\dim P-1$, we have $A(\dim P,k,t)=0+a_1t+\cdots+a_{\dim P}t^{\dim P}$ where\\
$a_0=0\leqslant a_1\leqslant\cdots\leqslant a_{ \dim P/2 +1} \geqslant a_{ \dim P/2 +2}\geqslant\cdots\geqslant a_{\dim P}$.
\end{enumerate}
Hence in $\delta(\lozenge P,t)= \delta_0+ \delta_1\,t + \cdots + \delta_{\dim P+1}\,t^{\dim P+1}$, we have 
$$\delta_0\leqslant \delta_1\leqslant\cdots\leqslant \delta_{\dim P/2}  \stackrel{\textnormal{ ? }}{\thicksim} \delta_{\dim P/2+1}\geqslant\cdots\geqslant \delta_{\dim P+1}.$$
So whatever (in)equality the question mark may represent, $\delta(\lozenge P)$ is unimodal.
\end{proof}

\begin{Ex}
In the above proof, when $\dim P$ is even, we were unable to say which inequality there is between $\delta_{\dim P/2}$ and $\delta_{\dim P/2+1}$. We remark that every scenario is possible. We give examples for $\dim P =4$.\\

When $P=\langle (0,0,0,0,1),(2,0,0,0,2),(0,1,0,0,1),(0,0,1,0,1),(1,1,1,2,1) \rangle$, then $\delta(\lozenge P)=(1,46,204,194,35,0)$.\\

When $P=\langle (0,0,0,0,1),(1,0,0,0,2),(0,1,0,0,1),(0,0,1,0,1),(1,1,1,2,1)\rangle$, then $\delta(\lozenge P)=(1,27,92,92,27,1)$.\\

And when $P=\langle (0,0,0,0,1),(1,0,0,0,1),(0,1,0,0,1),(0,0,1,0,1),(1,1,1,3,1) \rangle$, then $\delta(\lozenge P)=(1,28,118,158,53,2)$.
\end{Ex}

\begin{Ex}\label{voorbeeld}
It is not true in general that the $\delta$-vector of a lattice parallelepiped is alternatingly increasing. Simply take the parallelepiped spanned by $(0,0,1),(3,0,1)$ and $(0,1,1)$. Then $\delta(\lozenge P ) = (1,8,9,0)$. 
\end{Ex}

Now we show that the $\delta$-vector of a parallelepiped is alternatingly increasing if there is an interior lattice point.

\begin{Prop}
Suppose $\lozenge P$ is a lattice parallelepiped with at least one interior point. Then $\delta(\lozenge P)$ is alternatingly increasing.
\end{Prop}

\begin{proof}
We use the formula $\delta(\lozenge P, t) = \sum_{\emptyset\subseteq G \subseteq P} b(G)\,A(\dim P, \dim G, t)$. Since $\lozenge P$ has an interior point, $\deg (\lozenge P)= n+1$. Hence, we want to prove $\delta_0 \leqslant \delta_{n+1} \leqslant \delta_1\leqslant \cdots $. We notice that the coefficients of every $A(n,i,t)$ with $i \geqslant\lfloor(n/2)\rfloor $ exhibit such behaviour. The only problems can arise when there are nonzero $b(G)$ with $-1 \leqslant \dim G \leqslant \lfloor n/2\rfloor -1$. Since $\lozenge P$ has at least one interior point, we see that no problems can come from $A(n, -1,t)$. Because then there is also (at least one) $A(n,n,t)$ in the sum. Hence, it follows from Lemma~\ref{A(k,t)=delta(triangle line)} and symmetry that $A(n, -1,t)+A(n,n,t)$ has symmetric and unimodal coefficients, in particular they are alternatingly increasing.\\
From Lemma~\ref{lemma 2} we know that for every $A(n,i,t)$ in the sum with $0\leqslant i \leqslant \lfloor n/2\rfloor -1$ there will also be an $ A(n, n-1-i+k)$ with $0\leqslant k\leqslant i+1$. Therefore it is enough to prove that every $A(n,i,t)+A(n, n-1-i+k)$ has alternatingly increasing coefficients. From Lemma ~\ref{mirror with} we know that $A(n, i ,t)$ is the mirror image of $A(n, n-1-i)$. Then, using Lemma~\ref{A(k,t)=delta(triangle line)}(3) we can write $A(n,i,t)+A(n, n-1-i+k)$ in vector form as:
\begin{multline*} 
(0,2^{i+1}\delta_{n-1}(\triangle^{i+1} \times \ell^{n+1-2(i+1)})+2^{i+1-k}\delta_0(\triangle^{i+1-k} \times \ell^{n+1-2(i+1-k)}),
\\
\ldots,  2^{i+1}\delta_{0}(\triangle^{i+1} \times \ell^{n+1-2(i+1)})+2^{i+1-k}\delta_{n-1}(\triangle^{i+1-k} \times \ell^{n+1-2(i+1-k)}),0 ).
\end{multline*}
So, it is enough to prove for all $1\leqslant i\leqslant \lfloor n/2\rfloor $ and $0 \leqslant k \leqslant i$ that 
\begin{enumerate}
\item for $0 \leqslant j \leqslant \lfloor n/2\rfloor -1$:
	\begin{center}
		\begin{tabular}{l} 
	 	$2^k\delta_{n-1-j}(\triangle^{i}\times \ell^{n+1-2i})+ \delta_j(\triangle^{i-k} \times 	\ell^{n+1-2(i-			k)} )$\\
		$\quad\leqslant 2^k\delta_{j}(\triangle^{i}\times \ell^{n+1-2i})+ \delta_{n-1-j}(\triangle^{i-k} \times 				\ell^{n+1-2(i-k)} )$ and  
		\end{tabular}
	\end{center}
\item for $1 \leqslant j \leqslant \lfloor n/2\rfloor$:
	\begin{center}
		\begin{tabular}{l} 
	$2^k\delta_{j-1}(\triangle^{i}\times \ell^{n+1-2i})+ \delta_{n-j}(\triangle^{i-k} \times \ell^{n+1-2(i-k)} 	)$\\
	$\quad\leqslant 2^k\delta_{n-1-j}(\triangle^{i}\times \ell^{n+1-2i})+ \delta_{j}(\triangle^{i-k} \times 				\ell^{n+1-2(i-k)} ).\ \ \ $
		\end{tabular}
	\end{center}
\end{enumerate}

The second claim follows from the fact that every $\delta(\triangle^a\times \ell^{n+1-2a})$ is alternatingly increasing. To prove the first claim, it is enough to prove for $0 \leqslant j \leqslant \lfloor n/2\rfloor -1$ that $\delta_{j}(\triangle^{i-k} \times 	\ell^{n+1-2(i-k)})\leqslant 2^k \delta^{j}(\triangle^{i}\times \ell^{n+1-2i})$ and $2^k \delta_{n-1-j}(\triangle^{i}\times \ell^{n+1-2i}) \leqslant \delta_{n-1-j}(\triangle^{i-k} \times \ell^{n+1-2(i-k)} )$. This will follow from Corollary~\ref{cor1}. 
\end{proof}

\begin{Lem} For $1\leqslant i\leqslant (n+1)/2$ we have: 
\begin{align*}
\delta_{j}(\triangle^{i-1}\times \ell^{n+1-2(i-1)}) & \leqslant & 2 \delta_j(\triangle^{i}\times \ell^{n+1-2i} ) &\  when\ 0\leqslant j\leqslant \lfloor \frac{n-1}{2}\rfloor \textit{ and} \\
\delta_{j}(\triangle^{i-1}\times \ell^{n+1-2(i-1)}) &\geqslant & 2 \delta_j(\triangle^{i}\times \ell^{n+1-2i} )  &\ when\ \lfloor \frac{n-1}{2}\rfloor +1 \leqslant j \leqslant n-1.
\end{align*}
\end{Lem}

\begin{proof}
Let us look at the difference $2\delta(\triangle^i\times\ell^{n+1-2i},t)-\delta(\triangle^{i-1} \times \ell^{n+1-2(i-1)},t )$:
\begin{eqnarray*}
\lefteqn{2 \delta(\triangle^i\times\ell^{n+1-2i},t)-\delta(\triangle^{i-1} \times \ell^{n+1-2(i-1)},t )}\\
&=& (1-t)^{n+2} \sum_{m=0}^{\infty}t^m\Big[2 \binom{\scriptstyle{m+2}}{\scriptstyle{2}}^{i}(m+1)^{n+1-2i} - \binom{\scriptstyle{m+2}}{\scriptstyle{2}}^{i-1}(m+1)^{n+1-2(i-1)}   \Big]\\
&=&(1-t)^{n+2}\sum_{m=0}^{\infty} t^m\binom{\scriptstyle{m+2}}{\scriptstyle{2}}^{i-1}(m+1)^{n+2-2i}\\
&=&(1-t)\delta(\triangle^{i-1}\times \ell^{n-2(i-1)},t).
\end{eqnarray*}
Knowing this, the lemma follows from the fact that $\delta(\triangle^{i-1}\times \ell^{n-2(i-1)})$ is unimodal with top in $\lfloor(n-1)/2\rfloor$ (see Theorem~\ref{delta(triangle k line i)}).
\end{proof}

\begin{Cor}\label{cor1} 
For all $0\leqslant i\leqslant (n+1)/2$ and $0\leqslant k \leqslant i$ we have:
\begin{align*}
\delta_{j}(\triangle^{i-k}\times \ell^{n+1-2(i-k)}) & \leqslant & 2^k \delta_j(\triangle^{i}\times \ell^{n+1-2i} ) &\  when\ 0\leqslant j\leqslant \lfloor \frac{n-1}{2}\rfloor \textit{ and}\\
 \delta_{j}(\triangle^{i-k}\times \ell^{n+1-2(i-k)}) &\geqslant & 2^k \delta_j(\triangle^{i}\times \ell^{n+1-2i} ) &\ when\ \lfloor \frac{n-1}{2}\rfloor +1 \leqslant j \leqslant n-1.
\end{align*}
\end{Cor}

\begin{proof}
This follows immediately from the previous lemma.
\end{proof}

Suppose $P= \langle v_0,...,v_n\rangle$ is an $n$-dimensional simplex in $\R^{n+1}$ such that $b(P)\geqslant 1$. Say $v'=\sum_{i=0}^n \alpha_i v_i$ where $0<\alpha_i <1$, is such a lattice point in the box of $P$. Then we define a map 
\begin{eqnarray*}
\varphi: \bigcup_{\substack{G\subseteq P\\0\leqslant \dim G \leqslant \lfloor \frac{n}{2}\rfloor -1}} \bx(G)\cap\Z^{n+1} & \longrightarrow & \bigcup_{\substack{G\subseteq P\\n-\lfloor \frac{n}{2}\rfloor \leqslant \dim G}} \bx(G)\cap\Z^{n+1} \setminus \{v'\} \\
\sum_{i=0}^n \lambda_i v_i & \longmapsto &\sum_{i=0}^n \{\lambda_i+\alpha_i\} v_i,
\end{eqnarray*}
where $\{ \cdot \}$ denotes the fractional part of a real number.

\begin{Lem}\label{lemma 2} The map $\varphi$ defined above is injective and if $v\in \bx(G)\cap \Z^{n+1}$ where $\dim G =r$, then  $\varphi(v)\in \bx(G')\cap\Z^{n+1}$ where $G'$ is a face of $P$ with $\dim G' \geqslant n-1-r$.
\end{Lem}

The previous results inspire us to ask the following addendum to Question~\ref{Hoofdvraag}.

\begin{Quest}
Is it true that the $\delta$-vector of an integrally closed polytope with an interior lattice point is alternatingly increasing\,?
\end{Quest}

To conclude this section, we give a criterion for a lattice parallelepiped to be reflexive.

\begin{Prop} 
Set $P= \langle v_0,...,v_n\rangle$ in $\R^{n+1}$. Then $\lozenge P$ is a translate of a reflexive polytope if and only if 
$b(P)=1$ and $b(G)\leqslant 1$ for all other faces $G$ of $P$.
\end{Prop}

\begin{proof}
First suppose $\lozenge P$ is a translate of a reflexive polytope. Then by definition $b(P)$ must be equal to $1$ and this unique interior lattice point must be equal to $(v_0+ \cdots + v_n)/2$. We can also write down the equations for the facets of $\lozenge P$. Denote by $f_i(x_0,\ldots ,x_n)$ the determinant of the matrix whose $j$th column contains the coordinates of $v_j$ for $j\neq i$ and whose $i$th column contains $x_0,\ldots ,x_n$. Then all facets of $\lozenge P$ are given by
\[ F_i :=\{ f_i(x_0,\ldots ,x_n) = 0 \} \quad \text{or} \quad F'_i := \{ f_i(x_0,\ldots ,x_n) = \det(v_0, \ldots, v_n)  \}, \]
for $i=0,\ldots , n$. 
The unique interior point lies on every hyperplane given by the equation $f_i=\det(v_0,\ldots,v_n)/2$. The fact that $P$ is reflexive means that for every $0\leqslant i\leqslant n$ and every $0<\lambda< 1/2$ there can be no hyperplane $f_i=\lambda \det(v_0,\ldots,v_n)$ that contains a point with integer coefficients. Now suppose there is a face $G=\langle v_0,\ldots,v_r\rangle$ of $P$ that contains more than one lattice point in its box. Then there must be a point $v$ in $\bx(G)\cap\Z^{n+1}$ that can be written as $v=\sum_{i=0}^r \lambda_i v_i$ with $0<\lambda_i<1$ for all $i$ and for which at least one $\lambda_i$ is different from $1/2$. We can assume that $ 0< \lambda_i < 1/2$. Then this lattice point lies on the hyperplane $f_i=\lambda_i \det(v_0,\ldots,v_n)$, which is a contradiction.

Now suppose $b(P)=1$ and all other faces $G$ of $P$ have at most one lattice point in their box. To show that $\lozenge P$ is a translate of a reflexive polytope we show that $\delta(\lozenge P)$ is symmetric. It is easy to check that in this situation $b(\langle v_0,\ldots,v_r\rangle)=1$ if and only if $b(\langle v_{r+1},\ldots,v_n\rangle)=1$. This way we find for every $0\leqslant r\leqslant n$, that 
\[ \sum_{\substack{G\subseteq P\\ \dim G=r }}b(G)=\sum_{\substack{G\subseteq P\\ \dim G=n-1-r}} b(G).\] 
Then we use $\delta(\lozenge P, t)=\sum_{\emptyset \subseteq G\subseteq P}b(G)A(\dim P, \dim G, t)$ and Lemma~\ref{mirror with} to find that $\delta(\lozenge P)$ is symmetric.
\end{proof}

\section{Integrally closed polytopes of small dimension}\label{section3}

In this part of the paper we show that all integrally closed polytopes of dimension at most 4 have a unimodal $\delta$-vector. First we prove an elementary lemma.

\begin{Lem}\label{LemmaUnimodal}
Let $P$ be an integrally closed polytope of dimension $d$. If $\delta_1 = 0$, then $\delta_2 = \cdots = \delta_d = 0$ and $P$ is a unimodular simplex. 
\end{Lem}

\begin{proof}
From the definition of the $\delta$-vector one deduces that
\[  \sharp (nP) = \binom{n+d}{d} + \delta_1 \binom{n+d-1}{d} + \cdots + \delta_d \binom{n}{d}. \]
In particular, if $\delta_1 = 0$ then $\sharp P = d+1$ and $P$ is a simplex. Since $P$ is integrally closed, it follows that $\sharp (nP) \leqslant \binom{n + d }{ d }$. Taking $n = d$ in the above formula shows then that necessarily $\delta_2 = \cdots = \delta_d = 0$. Then the normalized volume of $P$ equals 1 and hence $P$ is a unimodular simplex.
\end{proof}

\begin{Prop}\label{PropUnimodal}
Let $P$ be a polytope of dimension $d\geqslant 3$ with an interior lattice point (i.e., $\delta_d\neq 0$). Then
\[  \delta_0 \leqslant \delta_d \leqslant \delta_1 \leqslant \delta_{d-1} \leqslant \delta_2. \]
\end{Prop}

\begin{proof}
If $\delta_d\neq 0$, then trivially $\delta_d \geqslant \delta_0 =1$. One has $\delta_1\geqslant \delta_d$ since $\delta_d$ equals the number of interior lattice points in $P$, $\sharp P = \delta_1 + d+1$, and $P$ has at least $d+1$ vertices.
By Hibi's lower bound theorem \cite{HibiLower} we have $\delta_1 \leqslant \delta_{d-1}$. Finally, if $d\geqslant 4$ then $\delta_2\geqslant \delta_{d-1}$ by \cite[Eq.\ (6)]{Stapledon}. 
\end{proof}

\begin{Cor}
Let $P$ be an integrally closed polytope of dimension $d\leqslant 4$. Then $\delta(P)$ is unimodal.
\end{Cor}

\begin{proof}
The case $d\leqslant 1$ is a trivial exercise. Every 2-dimensional lattice polytope admits a regular unimodular triangulation, and hence is integrally closed. In this case, all possible $\delta$-vectors are classified, essentially by Scott \cite{Scott}. They are all unimodal.

Next assume that $d=3$. If $\delta_3 = 0$, then $1 = \delta_0 > \delta_1 < \delta_2$ is excluded by Lemma~\ref{LemmaUnimodal}. If $\delta_3\neq 0$ then we apply Proposition~\ref{PropUnimodal}.

Finally, let $d=4$. If $\delta_4=\delta_3 = 0$, then $\delta(P)$ is unimodal by Lemma~\ref{LemmaUnimodal}. If $\delta_4 = 0$ and $\delta_3 \neq 0$, then we can apply \cite[Prop.\ 3.4]{Stanley} and Lemma~\ref{LemmaUnimodal} to see that $\delta_0 \leqslant \delta_1\leqslant \delta_2$. Hence $\delta(P)$ is unimodal. (In this case $\delta_3\leqslant \delta_2$ by \cite[Eq.\ (6)]{Stapledon}.) If $\delta_4\neq 0$ then we apply Proposition~\ref{PropUnimodal}.
\end{proof}



In general it is not true that $\delta_1 \geqslant \delta_s$ if $s$ is the degree of $P$, see Example~\ref{voorbeeld}.

\section{Integrally closed reflexive polytopes}\label{ReflexivePolytopes}

In this section we relate the unimodality question for reflexive polytopes to the existence of certain triangulations, which we call \textit{box unimodal triangulations}. First we need to discuss some preparatory work concerning triangulations.

Let $P$ be a reflexive polytope of dimension $d$ in $\mathbb{R}^d$ and let $\mathcal{T}$ be a triangulation of $\partial P$. Let $F$ be a face of $\mathcal{T}$ (the empty set is considered to be a face of $\mathcal{T}$ of dimension $-1$). The \textit{$h$-polynomial} of $F$ is
\[  h_F(t) = \sum_{F\subseteq G} t^{\dim G -\dim F}(1-t)^{d-1-\dim G}, \]
where we sum over all faces $G$ of $\mathcal{T}$ containing $F$. In fact this is the $h$-polynomial of a simplicial complex, namely of the \textit{link} of $F$ in $\partial P$. It is well known that $h_F(t)$ is a polynomial of degree $d' = d-1-\dim F$ with nonnegative integer coefficients. If $h_i$ denotes the coefficient of $t^i$ in $h_F(t)$ then
\[   h_i = h_{d' - i}.  \]
Moreover, if the triangulation is \textit{regular} then
\[   1 = h_0 \leqslant h_1 \leqslant \cdots  \leqslant h_{\lfloor d'/2\rfloor}, \]
i.e., the coefficients of $h_F(t)$ are unimodal. See for instance \cite[Lemma 2.9]{Stapledon} for a nice proof of these properties of $h_F(t)$. For the definition of a regular triangulation and a proof of the existence of such triangulations we refer to \cite[\S 1.F]{BrunsGubeladze}. 

Let $S$ be a lattice simplex in $\mathbb{R}^d$ of dimension $r$, with vertices $v_0,\ldots, v_r$. We embed $S$ as $\overline{S} = S\times \{1\}$ in $\mathbb{R}^{d+1}$ and we write $\overline{v_i}$ for the vertex $(v_i,1)$ of $\overline{S}$. Let $u : \R^{d+1} \to \R$ be the projection on the last coordinate. Then one defines the \textit{box polynomial} of $S$ as
\[ B_S(t) = \sum_{v\in \text{Box}(\overline{S})\cap \mathbb{Z}^{d+1}}  t^{u(v)}. \]
The coefficients of $B_S(t)$ are symmetric since we have the involution
\[  a_0 \overline{v_0} + \cdots + a_r \overline{v_r} \mapsto (1-a_0) \overline{v_0} + \cdots + (1- a_r) \overline{v_r} \]
on $\text{Box}(\overline{S})$. More precisely, if $b_i$ denotes the coefficient of $t^i$ in $B_S(t)$ then
\[  b_i = b_{r+1-i}.\]
We put $B_{\emptyset}(t) = 1$. Note that $B_S(t) = 0$ if $S$ is a unimodular simplex. 
Then we have the following theorem \cite[Thm.\ 1.3]{MustataPayne}, which is reminiscent of the Betke-McMullen formula \cite[Thm.\ 1]{BetkeMcMullen}.

\begin{Thm}
Let $P$ be a reflexive polytope and let $\mathcal{T}$ be a triangulation of the boundary $\partial P$. Then
\[  \delta(P,t) = \sum_{F\text{ face of }\mathcal{T}} B_F(t) h_F(t). \] 
\end{Thm}

We call a simplex $S$ \textit{box unimodal} if $S=\emptyset$ or if the coefficients $b_1,b_2,\ldots,b_{\dim S}$ of $B_S(t)$ are a unimodal sequence. We call a triangulation $\mathcal{T}$ of a lattice polytope or of the boundary of a lattice polytope \textit{box unimodal} if $\mathcal{T}$ is regular and if every face of $\mathcal{T}$ is box unimodal. The following corollary is immediate from the above.

\begin{Cor}\label{CorBoxUnimodal}
Let $P$ be a reflexive polytope such that $\partial P$ admits a box unimodal triangulation. Then $\delta(P)$ is unimodal.
\end{Cor}

\begin{Ex}\label{ExampleNotBoxUnimodal}
Let $S$ be the 5-dimensional simplex with vertices 
\[ (0,0,0,0,0), (1,0,0,0,0), (0,1,0,0,0), (0,0,1,0,0) , (0,0,0,1,0), (2,2,2,2,3).\]
Then $B_S(t) = t^2 + t^4$ since $\text{Box}(\overline{S}) \cap \Z^{6}  = \{(1, 1 ,1 ,1 ,1, 2),(2,2,2,2,2,4)\}$. Hence $S$ is not box unimodal. The idea of \cite{MustataPayne} is to use such a simplex as a facet of a reflexive polytope $P$ to obtain a nonunimodal $\delta$-vector $\delta(P)$. Since the simplex is \textit{empty}, i.e., it does not contain any lattice points besides its vertices, it is a face of every triangulation of $\partial P$, and hence $\partial P$ does not have any box unimodal triangulations.
\end{Ex}

We remark the following.

\begin{Prop}
Every lattice polytope $P$ of dimension $d\leqslant 4$ has a box unimodal triangulation.
\end{Prop}

\begin{proof}
Let $\mathcal{T}$ be a regular \textit{fine} triangulation of $P$, i.e., consisting of empty simplices. Since the coefficient $b_1$ of $B_S(t)$ is 0 if $S$ is an empty simplex, it follows that $S$ is box unimodal if $\dim S\leqslant 4$.  
\end{proof}

By using this kind of triangulations and Corollary~\ref{CorBoxUnimodal} we recover Hibi's result that every reflexive polytope of dimension $d\leqslant 5$ has a unimodal $\delta$-vector. Note that the simplex $S$ of Example~\ref{ExampleNotBoxUnimodal} is a 5-dimensional polytope without box unimodal triangulations. We ask the following question.

\begin{Quest}\label{BoxUnimodalTriangulation}
Does every integrally closed polytope have a box unimodal triangulation\,?
\end{Quest}

An affirmative answer to this question would prove that every integrally closed reflexive polytope has a unimodal $\delta$-vector. We think that this is a challenging problem, even for 5-dimensional polytopes. Let us illustrate this with a famous example.

\begin{Ex}
In \cite{BrunsGubeladzeUHC} Bruns and Gubeladze discuss an example of a 5-dimensional integrally closed polytope $P$ that is not covered by its unimodular subsimplices. It can be embedded in $\R^5$ as the polytope with vertices
\[ \begin{array}{c} (0,0,0,0,0),(0,1,0,0,0),(0,0,1,0,0),(0,0,0,1,0),(0,0,0,0,1) \\ 
(1, 0, 2, 1, 1) , (1, 2, 0, 2, 1), (1, 1, 2, 0, 2), (1, 1, 1, 2, 0), (1, 2, 1, 1, 2). \end{array} \]
There are 65 subsimplices of maximal dimension that are not unimodular. They have normalized volume 2 (60 of them) or 3 (5 of them). It is easy to see that every simplex of volume 2 is box unimodal. The simplices of volume 3 all have box polynomial $2t^3$. We conclude that all regular triangulations of $P$ are box unimodal\,!
\end{Ex}

\begin{Rem}
Let us sketch a possible intuitive explanation for the existence of box unimodal triangulations. Let $P$ be an integrally closed polytope of dimension $d$ in $\R^d$ and let $\mathcal{T}$ be a regular triangulation of $P$. Assume that $S$ is a face of $\mathcal{T}$ of dimension $r$ with nonzero box polynomial. Then there is a point $v\in \text{Box}(\overline{S})$. Let $v_0,\ldots, v_r$ be the vertices of $S$. Then 
\[  v = a_0 \overline{v_0} + \cdots + a_r \overline{v_r},  \tag{$\spadesuit$}\label{schoppen}  \]
with $0 < a_i < 1$ for all $i$, and with $\sum_i a_i = u(v)$. By symmetry, we may assume that $u(v) \leqslant (r+1)/2$. On the other hand, we can also embed $P$ as $\overline{P} = P\times \{1\}$ in $\R^{d+1}$, and since $P$ is integrally closed, there are lattice points $p_1,\ldots , p_{u(v)}$ in $P$ such that
\[  v = \overline{p_1} + \cdots + \overline{p_{u(v)}}.  \tag{$\nabla$}\label{nabla}  \]
Some of the $p_i$ might coincide, or some $p_i$ might coincide with some $v_j$, but from (\ref{schoppen}) and (\ref{nabla}) we can always deduce an equality of two convex combinations, supported on \textsl{disjoint sets} of lattice points of $P$. Let us illustrate this with a simple example. The integrally closed 5-dimensional polytope with vertices $p_0 = (0,0,0,0,0),p_1 = (1,0,0,0,0),p_2 = (0,1,0,0,0),p_3 = (0,0,1,0,0), p_4=(0,0,0,1,0) , p_5 = (-1,-1,-1,-1,3)$ and $p_6 = (0,0,0,0,1)$ can be triangulated by using $S_1 = \langle p_0,\ldots , p_5 \rangle$ and $S_2 = \langle p_1 , \ldots , p_6 \rangle$ as simplices of maximal dimension. This triangulation is not box unimodal because $B_{S_1}(t) = t^2+ t^4$. The procedure above applied to the point $(0,0,0,0,1,2) \in \text{Box}(\overline{S_1})$ leads to the equality
\[   \frac{1}{5}p_1 + \cdots + \frac{1}{5}p_5  =  \frac{2}{5}p_0 + \frac{3}{5}p_6. \tag{$\blacksquare$}\label{boxje}  \]
Such an equality of convex combinations supported on disjoint sets induces one or more \textit{circuits}. We refer to \cite{DeLoeraRambauSantos} for the definition of a circuit, and also for the explanation of the notions used below. Such circuits are used for performing \textit{flips} on the triangulation. In the example, we can change the triangulation to the one given by the simplices
\[ \begin{array}{c} \langle p_0,p_1,p_2,p_3,p_4,p_6 \rangle, \langle p_0,p_1,p_2,p_3,p_5,p_6 \rangle , \langle p_0,p_1,p_2,p_4,p_5,p_6 \rangle,\\ \langle p_0,p_1,p_3,p_4,p_5,p_6 \rangle \text{ and } \langle p_0,p_2,p_3,p_4,p_5,p_6 \rangle, \end{array} \]
i.e., we flip the axis of the triangulation from $\langle p_1,\ldots ,p_5 \rangle$ to $\langle p_0,p_6 \rangle$ as indicated by (\ref{boxje}). The resulting triangulation is unimodular. Hence the philosophy is the following: since $P$ is integrally closed, simplices of a given triangulation with nonzero box polynomial induce circuits which can be used to modify the triangulation. In this way, one might hope to get rid of non box unimodal simplices. The theory of \textit{secondary polytopes}, used to study all regular triangulations of $P$, might be very helpful for achieving this goal.

To conclude, we wonder whether the following very strong version of Question~\ref{BoxUnimodalTriangulation} holds.
\end{Rem}

\begin{Quest}
Let $P$ be an integrally closed polytope. Does there exist a regular triangulation $\mathcal{T}$ of $P$ such that all faces $S$ of $\mathcal{T}$ satisfy
\begin{itemize}
\item $B_S(t) = 0$ if $\dim S$ is even and
\item $B_S(t) = \alpha\, t^{(\dim S+1)/2}$ for some $\alpha\in \mathbb{Z}_{\geqslant 0}$ if $\dim S$ is odd\,?
\end{itemize}
\end{Quest}

\bibliographystyle{alpha}
\bibliography{./LatticeParallelepipeds}

\end{document}